\documentclass[10pt]{amsart}
\usepackage{amsfonts,amsthm}
\usepackage[usenames,dvipsnames]{color}

\newtheorem{theorem}{Theorem}[section]
\newtheorem{proposition}[theorem]{Proposition}

\newtheorem{remark}[theorem]{Remark}
\newtheorem{lemma}[theorem]{Lemma}
\newtheorem{definition}[theorem]{Definition}

\newcommand{\Prob}{{\mathbb{P}}}
\newcommand{\eps}{{\varepsilon}}

\title[Modulus of continuity of averages of SRB
measures]
{Modulus of continuity of averages of SRB
measures for a transversal family of piecewise expanding unimodal maps}
\author{Fabi\'an Contreras}

\begin{document}

\maketitle

\begin{abstract}
Let $f_t:[0,1] \to [0,1]$ be a family of piecewise expanding unimodal maps with a common critical point that is dense for almost all $t \in [a,b]$.  If $\mu_t$ is the corresponding SRB measure for $f_t$, we study the regularity of $\Gamma(t)=\int \phi d\mu_t$ when assuming that the family is transversal to the topological classes of these maps, more precisely, we prove that if $J_t(c)=\sum_{k=0}^{\infty} \frac{v_t(f_t^k(c))}{Df_t^k(f_t(c))} \neq 0$ for all $t$, where $v_t(x)=\partial_s f_s(x)|_{s=t}$, then $\Gamma(t)$ is not Lipschitz for almost all $t\in [a,b]$.  Furthermore, we give the exact modulus of continuity of $\Gamma(t)$.
\end{abstract}

\tableofcontents

\section{Introduction}

Sinai-Ruelle-Bowen (SRB) measures play important role in the study of statistical properties of dynamical systems.
Let $f$ be a map of a manifold $M$ preserving a measure $\mu.$
A point $x$ is called $\mu$-regular if for every continuous function $\phi$ we have
$$ \frac{1}{N} \sum_{n=0}^{N-1} \phi(f^n x)\to \mu(\phi). $$
A measure $\mu$ is called SRB if the set of $\mu$ regular points has positive Lebesgue measure.
In other words the SRB measure describes statistics of a Lebesgue positive measure set of initial conditions.

For example, if $f$ preserves an absolutely continuous invariant measure which is ergodic then that measure is SRB.

Another case where SRB measures are known to exist is when some hyperbolicity is present.

In particular, when the system is uniformly hyperbolic, for example, for topologically transitive Axiom A diffeomorphisms
or for smooth expanding maps SRB measures exists are unique and have good statistical properties such as the Central Limit
Theorem (CLT) \cite{V}. The CLT states that if $\phi$ is a Holder continuous function and $x$ is chosen uniformly with respect to
the Lebesgue measure then
$$ \lim_{N\to\infty} \Prob\left(\frac{\sum_{n=0}^{N-1} \phi(f^n x)-n\mu(\phi)}{\sigma(\phi,f)\sqrt{N}}\leq z\right)=
\int_{-\infty}^z \frac{1}{\sqrt{2\pi}} e^{-t^2/2} dt $$
where the diffusion coefficient $\sigma(\phi,f)$ is defined by
$$ \sigma^2(\phi, f)=\mu(\phi^2)+2\sum_{n=1}^\infty \mu(\phi (\phi\circ f^n)). $$
In case $\phi$ is smooth (rather than just Holder continuous) the normalizing constants
$$ \mu_{SRB}(\phi, f) \mbox{ and } \sigma^2(\phi,f) $$
depend smoothly on $f.$

This smoothness plays important role in averaging theory including some problems of statistical mechanics
\cite{Ru2, CD, D-AIM, DL}.

Uniformly hyperbolic systems appear rarely in applications. Much more common are systems which are either
nonuniformly hyperbolic on the set of large measure (notable examples are quadratic family \cite{Jak} and
Henon family \cite{BY})
or are hyperbolic but have singularities (notable examples are Lorenz system \cite{Tu}
and Lorentz gas \cite{CM}).

While uniformly hyperbolic system provides us with a good understanding on what happens for more general chaotic maps,
in the sense that many results first proven in the uniformly hyperbolic setting hold under much weaker conditions (see \cite{V})
the families of uniformly hyperbolic maps are not good models for predicting what happens with more general families.
Therefore our understanding of parameter dependence of invariant measures in weakly hyperbolic systems are quite poor.
To remedy this situation David Ruelle suggested to look at families of piecewise expanding unimodal maps.

We call a map $f:[0,1]\to [0,1]$ piecewise expanding unimodal map (PEUM) if there is a point $c$ and two maps
$f_L$ defined on $[0, c+\eps]$ and $f_R$ defined on $[c-\eps, 1]$ such that
$f_1(c)=f_2(c)$ and there is a constant $\lambda>1$ such that $|Df_*(x)|\geq \lambda$ for all $x$ from the domain of $f_*$ with $*=L,R$
and

\begin{equation}\label{ExpUnimod}
f(x) = \begin{cases} f_L(x) & \mbox{if } x\leq c\\
f_R(x) & \mbox{if } x\geq c. \end{cases}\end{equation}

PEUMs have unique absolutely continuous invariant measure \cite{LY} which is ergodic (in fact it is mixing and even exponentially mixing \cite{Bal2, V}) so it is the SRB measure for this system.

Several papers have been devoted to studying regularity of SRB absolutely continuous invariant measures in families of PEUMs. In particular some sufficient conditions for regular dependence of SRB
measures on parameters have been found (those conditions however are exceptional in the sense that they do not hold
for typical families).

This work also deals with families of PEUMs. That is, if $t\in [a,b]$, we work with a smooth one-parameter family of PEUMs $f_t: [0,1] \to [0,1]$, with $\mu_t$ the (unique) SRB measure associated to each $f_t$.  Thus, we want to study the regularity of $\Gamma(t)= \int \phi\; d\mu_t$, with $\phi$ in some suitable space.

In \cite{R1} and \cite{R2}, Ruelle considered the case $v= X \circ f$ and suggested a candidate for the derivative of $\Gamma(t)$. In \cite{Bal1}, Baladi studied properties of a complex function involved in Ruelle's candidate based on
spectral perturbation theory for transfer operators. She found a different way to express Ruelle's suggestions.
The latter was used by Baladi and Smania (\cite{BS1},\cite{BS2}) to give sufficient and necessary conditions for the differentiability of $\Gamma (t)$ (the differentiability does not always hold as shown in \cite{Mazz},\cite{Bal1},\cite{LiSm}) and exhibited an explicit formula for the derivative.  Besides some routine differentiability and irreducibility assumptions on the family $f_t$, the important assumption is that the quantity
$$J(c,f) = \sum_{k=0}^\infty \frac{v(f^{k}(c))}{Df^{k}(f(c))}$$ equals zero. $\;$

As we mentioned before, in \cite{Bal1}, Baladi shows that there is a one-parameter family when a transversality condition holds, or equivalently when $J(c,f)\neq 0$, and concludes that  $\Gamma(t)$ is not Lipschitz (and then it cannot be differentiable) for such a particular family.  In \cite{LiSm}, De Lima and Smania proved that for every $\Omega \subset [a,b]$ with positive Lebesgue measure, $\Gamma(t)$ is not Lipschitz, for almost all $t$ in $\Omega$.  Our approach is more elementary.  While in \cite{LiSm}, their analysis lie on studying a decomposition of the difference $\rho_{t+h}-\rho_t$ in the space of generalized bounded variation functions,  where $\rho_s$ is the Radon-Nikodym derivative of $\mu_s$, our analysis is focused on the phase space.  More precisely, we decompose the phase space $[0,1]$ into two parts: one part where we can emulate the well-known smooth expanding case in $S^1$ and a second part (the complement of the set before) where we try to estimate the exact modulus of continuity for $\Gamma(t)$, for almost all $t$ in $[a,b]$.

With the previous in mind, our main result is the following:\\

\begin{theorem}\label{MainResult} Suppose that $|J(c,f_t)|>\epsilon_1$ for some $\eps_1>0$ and for all $t \in [a,b]$.$\;$Then, for $\phi \in$ Lip$[0,1]$ and for almost all $t$, we have

$$\limsup_{h \downarrow 0} \frac{\Gamma(t+h)-\Gamma(t)}{h\sqrt{|\log(h)\log\log|\log(h)||}}=2\sqrt{2}\rho_t(c)J(c,f_t)\sigma_t(\phi)\bigg(\int \log |Df_t(x)|\;d\mu_t\bigg)^{-1/2}.$$\\

\end{theorem}

The paper is organized as follows:

In Section 2, we state the assumptions on our family of PEUMs and formulate the idea for proving Theorem \eqref{MainResult}.  Also, the section shows an example that fulfills all the requirement asked, the necessary definitions throughout the paper, and some results on the decomposition needed for the phase space.  Section 3 and 4 are focused on proving the results stated on Section 2.  Finally, Section 5 starts by stating some preliminary results needed for our main result and ends with the proof of Theorem \eqref{MainResult}.  At the end, In Section 6, we present the proof of the technical lemmas from the beginning of Section 5.

\section{Preliminaries}

\subsection{Setting and assumptions}
Recall that expanding unimodal maps are defined by formula (\ref{ExpUnimod}). Now we consider families of such maps.
Namely, we assume that $f_{L,t}(x)$ is defined for $(t,x)\in [a,b]\times [0, c+\eps]$ and
$f_{R,t}(x)$ is defined for $(t,x)\in [a,b]\times [c-\eps,1]$ and that $f_{*,t}(x)$ are $C^2$ functions of their arguments, with $*=L,R$.
Then we let
$$ f_t(x)=\begin{cases}f_{L,t}(x) & \mbox{if } x\leq c\\
                     f_{R,t}(x) & \mbox{if } x\geq c.  \end{cases} $$
Thus we assume that $c$ is a common critical point for all $s$. \\

We will also assume the following:

\begin{enumerate}
  \item $c$ is not periodic for almost all $t$.
  \item $f_t$ are uniformly expanding, i.e, there exists a constant $\lambda>1$ such that

    \begin{equation}\label{Lambda}
        |Df_{*,t}(x)|\geq \lambda,
    \end{equation}for all $t$ and $*=L,R$.
  \item $f_t$ is topologically mixing for all $t$.

  \item The family $\{f_t\}$ is transversal to the topological classes of these PEUMs, that is, there exists $\eps_1>0$ such that

    \begin{equation}\label{TransversalCondition}
        J_t(c)= \sum_{k=0}^{\infty} \frac{v_t(f_t^k(c))}{Df^k_t(f_t(c))} > \eps_1\footnote{Indeed, we could require that $J_t(c)\neq 0$, however, if $J_s(c)>0$ for some $s$, then $\inf\{J_t(c):t\in[a,b]\}>0$ by \cite{BS2} },
    \end{equation} for all $t$, where $v_t(x)=\frac{\partial  }{\partial t}f_t(x) $.

    \item For $\phi \in$ Lip$[0,1]$, the diffusion coefficients are positive for all parameters, i.e.,

    \begin{equation}
        \sigma_t(\phi)=\int \bigg(\phi-\int \phi d\mu_t \bigg)^2d\mu_t+2\sum_{k>0}\int \bigg(\phi-\int \phi d\mu_t \bigg)\bigg(\phi\circ f_t^k-\int \phi d\mu_t \bigg)d\mu_t > 0,
    \end{equation} for all $t$.

\end{enumerate}

By \cite{LY}, for each $t$, there exists a unique ergodic absolutely continuous invariant measure $\mu_t=\rho_t dx$ for $f_t$.  Let $\Gamma(t)=\Gamma(t,\phi)=\int \phi(x) \rho_t(x)dx$. As described before, the purpose is to study the modulus of continuity of $\Gamma(t)$.  For this, by uniform Lasota-Yorke estimates (see \cite{Bal2}), there exist $C\geq 1$ and $\delta \in (0,1)$ such that for all $n \geq 1$.

\begin{equation}\label{Inequality}
\left| \int \phi(x) \rho_{s}(x)\;dx - \int \phi(f_s^n x) \;dx     \right| \leq C\delta^n.
\end{equation}

Thus, we can work with iterations of the systems, that is, we can study the regularity of $\Gamma(t)$ by means of the following approximation

\begin{equation}\label{approximation}
\Gamma(t+h)-\Gamma(t)=\int \phi(f_{t+h}^n x) \;dx  - \int \phi(f_t^n x) \;dx+O(\delta^n).
\end{equation}

Therefore, the idea is then to take $n=n(h)$ depending on $h$ and in such a way that $n$ and $h$ are inversely proportional.

\subsection{Example}

A well-studied family of maps is the family of tent maps $f_t:[0,1] \to [0,1]$ defined by

$$f_t(x)=\left\{
           \begin{array}{ccc}
             tx &  \mbox{ if }& x\in[0,1/2] \\
             t(1-x) & \mbox{ if } & x\in[1/2,1] \\
           \end{array}
         \right.$$ for $t\in [\sqrt{2},2]$.  This family satisfies all the conditions above.  In fact, by \cite{BrMi}, the critical point $c=1/2$ is dense for almost all $t$, then $(1)$ is satisfied.  Condition $(2)$ is easy to see that holds, and $(3)$ and $(4)$ do as well by \cite{BalYng} and \cite{Ts} respectively. Condition $(5)$ can be assumed because if $\sigma_s(\phi)>0$ for some $s$ then, since $t \mapsto \sigma_t$ is continuous (\cite{Sch2}), $\sigma_t$ is positive in a neighborhood of $s$.

\subsection{Auxiliary Facts}

Set $t \in [a,b]$ and, as we observed, set $n=n(h)=\lfloor |\log(h)| \rfloor$, where $h>0$ is such that $t+h \in [a,b]$ and  $\lfloor |\log(h)| \rfloor$ denotes the largest integer less than $|\log(h)|$.

\begin{definition}
Let $x \in [0,1]$.  The nth-itinerary de $x$ with respect to the PEUM $f_t:[0,1] \to [0,1]$ is defined as the finite sequence

$$\omega_{t,n}(x)=(\sigma_0(x),\sigma_1(x),\sigma_2(x),\dots,\sigma_n(x))\in \{L,R\}^{n+1},$$ where

$$\sigma_i(x)=\left\{
                \begin{array}{ccc}
                  L &, & \mbox{if }\;f_t^i(x)\leq c.  \\
                  R &, & \mbox{if }\;f_t^i(x)> c. \\
                \end{array}
              \right.$$

The itinerary of $x$ with respect to $f_t$ is the sequence

$$\omega_{t}(x)=(\sigma_0(x),\sigma_1(x),\sigma_2(x),\dots,\sigma_n(x))\in \{L,R\}^{\mathbb{N}}.$$


\end{definition}
An important tool through this paper is the following operator.

\begin{definition}
  The transfer operator for a PEUM $f$ is defined as

    $$\mathcal{L}(\phi)(x)=\sum_{f(x)=y} \frac{\phi(y)}{|Df(y)|},$$ for all $\phi \in$ BV$[0,1]$.
\end{definition}

Let us recall that the space $BV[0,1]$ is a Banach space with the norm $\|\phi\|_{BV}=\|\phi\|_{\infty}+\mbox{ var}(\phi)$, where $\|\cdot\|$ is the usual supremum norm and var$(\cdot)$ is the total variation.

Also, since $f_t$ is stably mixing, there exists a constant $0<\theta<1$ (not depending on $t$), such that

\begin{equation}\label{StableSpectralGap}
  \mathcal{L}_t^n(h)=\rho_t\int h + O(\theta^n\|h\|_{BV}).
\end{equation}  Note that, in particular, $\rho_t$ is bounded by a constant not depending on $t$.

In the case of smooth expanding maps on the circle, the key to prove the differentiability of $\Gamma(t)$ is that for each $x \in [0,1]$, there exists $y \in [0,1]$  shadowing $x$, i.e., if $n \geq 0$, then $\omega_{t+h,n}(x)=\omega_{t,n}(x)$ and $f_{t+h}^n(x)=f_t(y)$.\\

With the previous in mind, with $n$ fixed, we shall decompose the integral $\Gamma(t+h)-\Gamma(t)$ in two parts:  one part where we can emulate the smooth expanding and another part which will be the corresponding complement.  For this, let us define the following

\begin{definition}
Let $n\geq 0$ and $0<h\ll 1$.  In $[0,1]$, we define the following sets

    $$A_{h,n}=\{x \in[0,1] \;:\; \mbox{ there exists } y\in[0,1] \mbox{ such that } \omega_{t+h,n}(x)=\omega_{t,n}(y) \mbox{ and } f^n_{t+h}(x)=f_t^n(y)   \} ,$$
    $$ B_{h,n}=\{y \in[0,1] \;:\; \mbox{ there exists } x\in[0,1] \mbox{ such that } \omega_{t+h,n}(x)=\omega_{t,n}(y) \mbox{ and } f^n_{t+h}(x)=f^n_t(y)   \}. $$\\
     If $x\in A_{h,n}$, define $y_n(x)$ as the corresponding $y \in [0,1]$ in the definition of $A_{h,n}$.

\end{definition}

\begin{definition}
  If $x \in A_{h,n}$ and $y=y_n(x)$, define

        $$J_{t,k}(y)=J_{k}(y)= -\sum_{j=0}^{k-1} \frac{v_t(f_t^j(y))}{Df_{t}^j(f_t(y)) },$$ where $v_t=\partial_s f_s|_{t=s}$.

\end{definition}

Note that $J_k(c)$ converges to $J_t(c)$ as $k \to \infty$.

The properties of the sets $A_{h,n}$ and $B_{h,n}$ are described in the following lemma.

\begin{lemma}\label{lemmaBtn}
(a) the complement of $A_{h,n}$ equals
\begin{equation}\label{complementAtn}
[0,1]\backslash A_{h,n}= \displaystyle \bigcup_{k=0}^{n} f_{t+h}^{-k}I_{n-k}
\end{equation}
where
$$I_k=\begin{cases}
[c+h\frac{J_k(c)}{Df_{L,t}(c)},c]+O(h^{2}) & \text{if }J_k(c)\leq 0\\
 [c, c-h\frac{J_k(c)}{Df_{R,t}(c)}]+O(h^{2}), & \text{if }J_k(c)>0.
\end{cases} $$ and where $Df_{L,t}(c)$ and $Df_{R,t}(c)$ are the derivative of $f_{L,t}$ and $f_{R,t}$ respectively at $x=c$.

Moreover, $$\bigg|[0,1]\backslash A_{t,n}\bigg|=O(tn),$$ where $|\cdot |$ denotes the Lebesgue measure.

(b) The complement of $B_{h,n}$ equals
\begin{equation}\label{complementBtn}
[0,1]\backslash B_{h,n}= \displaystyle \bigcup_{k=0}^{n} f_t^{-k}\widetilde{I}_{n-k}
\end{equation}
with
$$\widetilde{I}_k=\begin{cases}
[c-h\frac{J_k(c)}{Df_{L,t}(c)},c]+O(h^{2})  & \text{if }J_k(c)>0\\
[c, c+h\frac{J_k(c)}{Df_{R,t}(c)}]+O(h^{2}) & \text{if} J_k(c) \leq 0
\end{cases} $$

Moreover, $$\bigg|[0,1]\backslash B_{h,n}\bigg|=O(tn).$$

\end{lemma}

Consider the following decomposition

\begin{equation}
\label{intphit}
\int \phi(f^n_{t+h}x)dx=\int_{A_{h,n}} \phi(f^n_{t+h}x)dx
+\int_{[0,1]\backslash A_{h,n}} \phi(f^n_{t+h}x)dx.
\end{equation}

We start with analyzing the first term in \eqref{intphit}. By definition of $A_{h,n}$ we have
$$ \int_{A_{h,n}} \phi(f^n_{t+h}x)dx=\int_{B_{h,n}} \phi(f_t^n y) \left(\frac{dx}{dy} \right)
dy. $$

\begin{lemma}
\label{LmChVar}
$$ \left(\frac{dx}{dy} \right)=1-hR_{t,n}(y)+O(h^2) $$
where
$$R_{t,n}(y)=\sum_{k=0}^{n-1}\bigg\{ \frac{v_t'(f_t^ky)}{Df_t(f_t^ky)}-\frac{v_t(f_t^ky)}{Df_t(f_t^ky)} \sum_{j=0}^k \frac{\xi(f_t^jy)}{Df_t^{k-j}(f_t^jy)} \bigg\}$$ and $\xi(z)=\frac{D^2f_t(z)}{Df_t(z)}$.
\end{lemma}

Accordingly
\begin{eqnarray*}
\int_{A_{h,n}} \phi(f^n_hx)dx&=&\int_{B_{h,n}} \phi(f_t^n y) dy-h \int_{B_{h,n}} \phi(f_t^n y) R_{t,n}(y) dy+O(h^2 n^2) \\
&=&\int_{B_{h,n}} \phi(f_t^n y) dy-h \int_0^1 \phi(f_t^n y) R_{t,n}(y) dy+O(h^2 n^2) \\
\end{eqnarray*}
where the last step uses Lemma \ref{lemmaBtn}(a).
It follows that
$$ \Gamma(t+h)-\Gamma(t)=-h \int_0^1 \phi(f_t^n y) R_{t,n}(y) dy+
\int_{[0,1]\backslash A_{h,n}} \phi(f^n_{t+h}x)dx-
\int_{[0,1]\backslash B_{h,n}} \phi(f^n_t y)dy+O(h^2 n^2). $$ \\

\begin{proposition}\label{IntegralRn} The integral
    $$\int \phi(f_t^n)R_{t,n}dx$$ is bounded by a constant that does not depend on $n$.
\end{proposition}

\begin{remark}\label{AssumptionZeroMean}
  Note that if $\bar{\phi}=\phi - \int \phi d\mu_t$, then $\bar{\phi}$ is of zero mean with respect to $\mu_t$ (i.e. $\int \bar{\phi} d\mu_t=0$) and $\Gamma(t+h,\phi)-\Gamma(t,\phi)=\Gamma(t+h,\bar{\phi})-\Gamma(t,\bar{\phi})$. Therefore, without loss of generality, we will assume that $\phi$ is of zero mean from now on.
\end{remark}

\section{Shadowable points.}

\begin{proof}[Proof of Lemma \ref{lemmaBtn}]
(a) Let $x\in A_{h,n}$ and let $y_n(x)$ the corresponding $y$ according to the definition of $A_{h,n}$. $\;\;$Then, $\omega_{t+h,n}(x)=\omega_{t,n}(y_n(x))$.$\;\;$The latter condition is the same as saying that, given $0\leq k \leq n$, $f^k_{t+h}(x)$ and $f_t^k(y_n(x))$ are both in either $[0,c]$ or $[c,1]$.

\vspace{.3cm}

Observe that if $s \geq 1$ and $z \in A_{h,n}$ then by chain rule

\begin{equation}\label{chainrule}
z=y_s(z)+h\frac{J_s(y_s(z))}{Df_t(y_s(z))}+O(h^2).
\end{equation}

Thus, we could express $y_s(z)$ as

\begin{equation}\label{ExpressingY(z)}
y_s(z)+O(h^2)=z-h\frac{J_s(y_s(z))}{Df_t(y_s(z))}.
\end{equation}

Note that if $0\leq k\leq n$, then

$$ f_t^{n-k}(f_t^k(y_n(x)))=f_t^n(y_n(x))=f^n_{t+h}(x)=f^{n-k}_{t+h}(f^k_{t+h}x)$$

Hence, $y_{n-k}(f^k_{t+h}(x))=f_t^k(y_n(x))$ and using $\eqref{ExpressingY(z)}$, we have that

\begin{equation}\label{DescriptionOffky1}
f_t^k(y_n(x))+O(h^2)=f^k_{t+h}(x)-h\frac{J_{n-k}(f_t^k(y_n(x))}{Df_t(f_t^ky_n(x))}.
\end{equation}

\vspace{.3cm}




Since $x$ and $y_n(x)$ have the same itinerary under $f_{t+h}$ and $f_t$ respectively up to the $n-$iteration, $f_t^k(y_n(x))$ and $f_{t+h}^k(x)$ must be sufficiently far away from $c$ so that they can be in the same side, for each $0 \leq k \leq n$.  More precisely, in order to have that $f_t^ky_n(x)$ and $f_{t+h}^kx$ stay both in the same side, it is sufficient and necessary to require that

\begin{eqnarray}\label{ConditionLeft}
f_{t+h}^k(x)\notin [c+h\frac{J_{n-k}(c)}{Df_{L,t}(c)}+O(h^2),c]&,&\mbox{ if }J_{n-k}(c)\leq 0.
\end{eqnarray}

or

\begin{equation}\label{ConditionRight}
f_{t+h}^k(x)\notin [c,c-h\frac{J_{n-k}(c)}{Df_{L,t}(c)}+O(h^2)],\mbox{ if }J_{n-k}(c)>0.\\
\end{equation}

Indeed, suppose $J_{n-k}(c)\leq 0$ and $f_{t+h}^k(x) \in [c+h\frac{J_{n-k}(c)}{Df_{L,t}(c)}+O(h^2),c]$.$\;$

\vspace{.3cm}

Since $f^k_{t+h}(x)=c+O(h^2)$, $f_t^k(y_n(x))=c+O(h^2)$, and since $\frac{J_{n-k}(x)}{Df_t(x)}$ is left $C^1-$continuous, we can replace $\eqref{DescriptionOffky1}$ by

\begin{equation}\label{DescriptionOffky2}
f^k_t(y_n(x)) + O(h^2)=f_{t+h}^k(x)-h\frac{J_{n-k}(c)}{Df_{L,t}(c)}.
\end{equation}

Now, we are assuming that $f_{t+h}^k(x)$ is in $[c+h\frac{J_{n-k}(c)}{Df_{L,t}(c)}+O(h^2),c]$, so in particular, $$f_{t+h}^k(x)>c+h\frac{J_{n-k}(c)}{Df_{L,t}(c)}+O(h^2).$$

Using $\eqref{DescriptionOffky2}$, the above inequality implies that $f_t^k(y_n(x))>c$, that is, $f_{t+h}^k(x)$ and $f^k(y)$ are in different side so they have different itineraries.  $\;$Therefore, if we want them to lie in the same side we must require the condition $(\ref{ConditionLeft})$ as we claimed. $\;$The condition $(\ref{ConditionRight})$ is proved similarly.

\vspace{.3cm}

Therefore,

$$A_{h,n}=[0,1]\backslash \bigcup_{k=0}^{n} f_{t+h}^{-k}I_{n-k}$$ and thus we have proved $(\ref{complementAtn})$.

Let us prove that the Lebesgue measure of  $[0,1]\backslash \displaystyle \bigcup_{k=0}^{n} f_{t+h}^{-k}I_{n-k}$ is of order $O(hn)$.
We have

$$\bigg| \bigcup_{k=0}^{n-1} f_{t+h}^{-k}I_{n-k}\bigg| \leq \sum_{k=0}^{n-1}|f_{t+h}^{-k}I_{n-k}|
=\sum_{k=0}^{n-1} \int_{f_{t+h}^{-k}I_{n-k}}dx$$
$$=\sum_{k=0}^{n-1} \int_{I_{n-k}} \mathcal{L}_{t+h}^k(1)(x)dx=\sum_{k=1}^{n-1} \int_{I_{n-k}} (\rho_{t+h}(y)+O(\theta^k))(x)dx$$
$$\leq \sum_{k=0}^{n-1}|I_{n-k}|O(1)\leq \sum_{k=1}^{n-1}O( h J_{n-k}(c))+O(h^2)$$
$$= O\bigg(h \sum_{k=0}^{n-1} J_{n-k}(c)\bigg)+O(h^2n)\leq O(hn)$$

(note that we used  \eqref{StableSpectralGap}).
Therefore, we conclude that

\begin{equation}\label{measurecomplement}
\bigg|\displaystyle\bigcup_{k=0}^{n} f_{t+h}^{-k}I_{n-k}\bigg|=O(hn)
\end{equation}

Similarly, we prove (b).

\qedhere

\end{proof}

\section{Changing variables}

\begin{proof}[Proof of Lemma \ref{LmChVar}]

By definition of $A_{h,n}$, given $x\in A_{h,n}$, there is $y=y_n(x)$ such that $f_{t+h}^n(x)=f_t^n(y_n(x))$.

Since $\frac{\partial y_n}{\partial x} = \displaystyle \prod_{k=0}^{n-1} \frac{\partial y_{k+1}}{\partial y_k}$ (note that $y_0=x$),
we analyze the factors of this product.
We have

$$
\frac{\partial y_n}{\partial x} = \prod_{k=0}^{n-1}(1+h\beta_k)
= 1+h\sum_{k=0}^{n-1} \beta_k +O(h^2 n^2)$$
where $\beta_k(y)=\displaystyle \sum_{k=0}^{n-1}\bigg\{ \frac{v_t'(f_t^ky)}{Df_t(f_t^ky)}-\frac{v_t(f_t^ky)}{Df_t(f_t^ky)} \sum_{j=0}^k \frac{\xi(f_t^jy)}{Df_t^{k-j}(f_t^jy)} \bigg\}.$
Hence,

$$\bigg(\frac{\partial y_n}{\partial x} \bigg)^{-1}
= 1-h\sum_{k=0}^{n-1}\beta_k +O(h^2n^2) \qedhere$$
\end{proof}

\section{Contribution of shadowable points}
\begin{proof}[Proof of Lemma \ref{IntegralRn}]

Let us write the integral as

\begin{eqnarray*}
\int \phi(f_t^ny)R_{t,n}(y)\;dy &=& \int \phi(f_t^ny)\bigg(\sum_{k=0}^{n-1}\bigg\{ \underbrace{\frac{v'_t(f_t^ky)}{Df_t(f_t^ky)}}_{(I)}-\underbrace{\frac{v_t(f_t^ky)}{Df_t(f_t^ky)} \sum_{j=0}^k \frac{\xi(f_t^jy)}{Df_t^{k-j}(f_t^jy)} }_{(II)}\bigg\}\bigg)\;dy.\\
\end{eqnarray*}

Let us start by analyzing $(I)$. $\;$Making the change of variables $z=f_t^{n-j}y$ and $w=f^jz$, we get

\begin{eqnarray*}
\sum_{k=0}^{n-1} \int \phi(f_t^ny)\frac{v'_t(f_t^ky)}{Df_t(f_t^ky)}\;dy &=& \sum_{j=1}^{n} \int \phi(f_t^ny)\frac{v'_t(f_t^{n-j}y)}{Df_t(f_t^{n-j}y)}\;dy\\
&=&\sum_{j=1}^{n} \int \phi(f_t^{j}z)\frac{v'_t(z)}{Df_t(z)} \mathcal{L}_t^{n-j}(1)(z)\;dz\\
&=&\sum_{j=1}^{n} \int \phi(w)\mathcal{L}_t^{j}(\frac{v'_t}{Df_t}\mathcal{L}_t^k(1))(w)\;dw.
\end{eqnarray*}

By Remark \ref{AssumptionZeroMean} and \eqref{StableSpectralGap}, we have that

\begin{eqnarray}
\bigg|\sum_{j=1}^{n} \int \phi(w)\mathcal{L}_t^{j}(\psi_t )(w)\;dw\bigg| &\leq& \sum_{j=1}^{n} O(\theta^j \|\frac{v'_t}{Df_t}\mathcal{L}_t^k(1)\|_{BV})\\
&\leq& \sum_{j=1}^{n} O(\theta^j \|\frac{v'_t}{Df_t}\|_{BV}\|\mathcal{L}_t^k(1)\|_{BV}).
\end{eqnarray}

By \eqref{StableSpectralGap}, $\|\mathcal{L}_t^k(1)\|_{BV}$ is bounded by $1+\|\rho_t\|_{BV}$.  Since $\sum_{j=1}^{n} \theta^j$ is bounded by $1/(1-\theta)$, it follows that $|\sum_{k=0}^{n-1} \int \phi(f_t^ny)\frac{v'_t(f_t^ky)}{Df_t(f_t^ky)}\;dy |$ is bounded by a constant not depending on $n$.\\

Now, let us study $(II)$.  $\;$ Making the change of variables $z=f_t^jy$ and $w=f_t^{k-j}(z)$, we obtain that

\begin{eqnarray*}
\int \sum_{k=0}^{n-1}\phi(f_t^ny)\frac{v_t(f_t^ky)}{Df_t(f_t^ky)}\sum_{j=0}^k \frac{\xi(f_t^jy)}{Df_t^{k-j}(f_t^jy)}&=& \sum_{k=0}^{n-1}\sum_{j=0}^{k} \int \frac{\phi(f_t^ny)v_t(f_t^{k-j}(f_t^jy)) \xi(f_t^jy) }{Df_t(f_t^{k-j}(f_t^jy))Df_t^{k-j}(f_t^jy)} dy\\
&=& \sum_{k=0}^{n-1} \sum_{j=0}^k  \int \frac{\phi(f_t^{n-j}z)v_t(f_t^{k-j}(z))\xi(z) }{Df_t(f_t^{k-j}(z))Df_t^{k-j}(z)} \mathcal{L}_t^j(1)(z)dz\\
&=&\sum_{k=0}^{n-1} \sum_{j=0}^k  \int \frac{\phi(f_t^{n-k}w)v_t(w) }{Df_t(w)} \mathcal{L}_{t,1}^{k-j}( \xi\cdot\mathcal{L}_t^j(1))(w)\;dw \\
&=& \sum_{k=0}^{n-1} \sum_{j=0}^k \int \phi(q)\mathcal{L}_t^{n-k}\bigg( \frac{v_t}{Df_t} \mathcal{L}_{t,1}^{k-j}( \xi\cdot\mathcal{L}_t^j(1)) \bigg)dq \\
\end{eqnarray*}

where $\mathcal{L}_{t,1}^i(\varphi)(w)=\sum_{f_t^i(z)=w} \frac{\varphi(z)}{Df_t^i(z)|Df_t^i(z)|}$.\\

By \cite{CoDo}, there exists $\bar{\lambda}>1$ such that

\begin{equation}\label{BVnorm}
\| \mathcal{L}_{t,1}^i h\|_{BV} =O(\bar{\lambda}^{-i}\|h\|_{BV}).
\end{equation}

Then, by Remark \ref{AssumptionZeroMean} and \eqref{StableSpectralGap}, we have that

\begin{eqnarray*}
\bigg| \sum_{k=0}^{n-1} \sum_{j=0}^k \int \phi(q)\mathcal{L}_t^{n-k}\bigg( \frac{v_t}{Df_t} \mathcal{L}_{t,1}^{k-j}( \xi\cdot\mathcal{L}_t^j(1)) \bigg)dq\bigg| &\leq& \sum_{k=0}^{n-1} \sum_{j=0}^k O(\theta^{n-k} \|\frac{v_t}{Df_t}\|_{BV}\| \mathcal{L}_{t,1}^{k-j}( \xi\cdot\mathcal{L}_t^j(1))\|_{BV}).\\
&\leq& \sum_{k=0}^{n-1} \sum_{j=0}^k O(\theta^{n-k} \bar{\lambda}^{-(k-j)} \|\mathcal{L}_t^j(1)\|_{BV} \|\frac{v_t}{Df_t}\|_{BV}\|\|\xi\|_{BV}).
\end{eqnarray*}

As observed before, $\|\mathcal{L}_t^j(1)\|_{BV}$ is bounded and since $\sum_{k=0}^{n-1} \sum_{j=0}^k \theta^{n-k} \bar{\lambda}^{-(k-j)}$ is bounded above, it follows that $|\int \sum_{k=0}^{n-1}\phi(f_t^ny)\frac{v_t(f_t^ky)}{Df_t(f_t^ky)}\sum_{j=0}^k \frac{\xi(f_t^jy)}{Df_t^{k-j}(f_t^jy)}|$ is bounded by a constant not depending on n.\\

Therefore, the integral $\int \phi(f_t^ny)R_{t,n}(y)\;dy$ is bounded by a constant that does not depend on $n$ as claimed.

\end{proof}

\section{Modulus of continuity of $\Gamma(t)$}

\label{ScIMNonSmooth}

\subsection{Preliminary Results}
We have already decompose the phase space in order to write $\Gamma(t+h)-\Gamma(t)$ as 

$$\Gamma(t+h)-\Gamma(t)=h\int \phi(f_t^n(y)) R_{t,n}(y)dy + \int_{[0,1]\backslash A_{h,n}} \phi(f_{t+h}^nx)dx-\int_{[0,1]\backslash B_{h,n}} \phi(f_t^ny)dy+O(h^2n^2).$$

Lemma \ref{IntegralRn} gives us a control over the integral $\int \phi(f_t^n(y)) R_{t,n}(y)dy$, so we need to still analyze the difference $\int_{[0,1]\backslash A_{h,n}} \phi(f_{t+h}^nx)dx-\int_{[0,1]\backslash B_{h,n}} \phi(f_t^ny)dy$. The idea to do this will be to estimate this difference by integrals over a set nearby the critical point in order to prove Theorem \ref{MainResult}.  To achieve this, we will need a couple of preliminary results.\\

Define $c_j(t)=f^j_t(c)$. Then, we start with this proposition about recurrence of points in the orbit of $c$.

\begin{proposition}
    There exists $m >1$ such that for almost all $t$ in a small interval around 0

    \begin{equation}\label{recurrence}
    |c_j(t)-c| > j^{-m}
    \end{equation} if $j$ is sufficiently large.

\end{proposition}

\begin{proof}
By regularity of $f^n_t$ on $t$, $c_n(t)$ is of bounded variation. Let $C$ be the constant that bounds the quotient of derivatives with respect to $t$ of $c_n(t)$.

Take $k_0$ such that $\Lambda^{k_0} > 2C$.

Assume first that

\begin{equation}
c_k(t)\neq c,
\end{equation} for all $t \in I$ and all $k \leq k_0$. $\;$Define

 $$w_{n}(t) = \{ s\; :\; c_{j}(t) \mbox{ and } c_{j}(s) \mbox{ have the same itinerary for } j \leq n \},$$

$$W_{n}(t) = \{c_{n}(s)\}_{s \in w_{n}}, \mbox{ and }$$

$$\Gamma_{n}(t) = d\left( c_{n}(t), \partial W_{n}(t) \right)$$

Let us also define $Z_{n} = \sup_{0<\epsilon<1} \frac{Leb(\{t:\Gamma_{n}(t)<\epsilon\})}{\epsilon},$ where $\mbox{Leb}(\cdot)$ denotes the Lebesgue measure.

We claim that there exists $\widetilde{K}$ such that

\begin{equation}\label{Zn}
Z_{n} <  \widetilde{K}.
\end{equation}

Assuming (\ref{Zn})
take $\epsilon = ṇ^{-m}$, for $n>1$ and $m>1$, then

\begin{eqnarray*}
\mbox{Leb}(t: |c_{n}(t) - c| < n^{-m})&\leq& \mbox{Leb}(\{t: |\Gamma_{n}(t)| < n^{-m}\})\\
&\leq& \widetilde{K} n^{-m}
\end{eqnarray*}

Then, $\mbox{Leb}(\{t: |c_{n}(t) - c| < n^{-m}\}) \leq \widetilde{K} n^{-m}$, which implies that

$$\sum^{\infty}_{n=1} \mbox{Leb}(\{t: |c_{n}(t) - c| < n^{-m}\}) < \infty $$

Therefore, we can by Borel-Cantelly lemma, there exist $n_0$ such that for all $n \geq n_0$

$\mbox{Leb}(\{t: |c_{n}(t) - c| < n^{-m}\})= 1$ as we want.

Hence, we need to prove (\ref{Zn}).

In fact, we prove that there exist $n_{0}>1$, $\vartheta < 1$ and $M>0$ such that

\begin{equation}\label{Znk}
Z_{n+k} \leq Z_{n}\vartheta + M.
\end{equation}

This will certainly imply (\ref{Zn}). In order to prove (\ref{Znk}), let us pick $\widetilde{\delta} << 1$. Then, we will analyze two cases\newline

(1) The components of $W_n$ that have measure less than $\widetilde{\delta}$.

(2) The components of $W_n$ that have measure greater than $\widetilde{\delta}$.

Let us work in the first case and let $V_n$ be a component of $W_n$ such that $|V_n| < \widetilde{\delta}$ and let $v_n$ the component of $\omega_n$ associated to $V_n$.$\;$Then, $f^{k_0}$ maps $V_n$ into, at most, two intervals contained in $W_{n+k_0}=\bigcup_{t}W_{n+k_0}(t).$ If $V_n$ is split then $V_n$ passes trough $c$ at some point.
Note that we cannot have more than two intervals because of (\ref{recurrence}).
Suppose we have two intervals. Let us call them $V'_{n+k_0}$ and $V''_{n+k_0}$.

\vspace{.3cm}

Take $\epsilon>0$ and note that by the expansivity of $f^{k_0}$, we have that $\mbox{Leb}(\{t\;:\;c_{n+k_0}(t)\in V'_{n+k_0}\cup V''_{n+k_0} \mbox{ and } \Gamma_{n+k_0}(t)<\epsilon \})$ is less or equal than $\mbox{Leb}(\{t\;:\;d(c_n(t),a)\leq \frac{\epsilon}{\lambda^{k_0}} \mbox{ or } \Gamma_n(t) \leq \frac{\epsilon}{\lambda^{k_0}}\}),$ where $a$ is the point where $V_n$ reaches $c$ at some point, this is, $f^j_{s_0}(a)=c$ for some $j\leq k_0$ and $s_0 \in V_n$.$\;$ By bounded distortion, the measure of $\{t\;:\;d(c_n(t),a)\leq \frac{\epsilon}{\lambda^{k_0}} \} $ is comparable to $\{t\in v_n : \Gamma_n(t) \leq \frac{\epsilon}{\lambda^{k_0}}\}$, then

\vspace{.3cm}
$ \mbox{Leb}(\{t\;:\;c_{n+k_0}(t)\in V'_{n+k_0}\cup V''_{n+k_0} \mbox{ and } \Gamma_{n+k_0}(t)<\epsilon \}) \leq$
\begin{eqnarray*}
&\;&\hspace{3cm}\leq \mbox{Leb}(\{t\;:\;d(c_n(t),a)\leq \frac{\epsilon}{\lambda^{k_0}} \mbox{ or } \Gamma_n(t)\leq \frac{\epsilon}{\lambda^{k_0}}\})\\
&\;&\hspace{3cm}\leq 2C\mbox{Leb}(\{t \;:\; \Gamma_n(t)\leq \frac{\epsilon}{\lambda^{k_0}}\})\\
\end{eqnarray*}

By summing over all components of $W_n$ with measure less than $\widetilde{\delta}$, we have that

$$\mbox{Leb}(\{t\;:\;c_{n+k_0}(t)\in V'_{n+k_0}\cup V''_{n+k_0} \mbox{ and } \Gamma_{n+k_0}(t)<\epsilon \}) \leq \frac{2C\epsilon}{\lambda^{k_0}}Z_n,$$ (note that we use the definition of $Z_n$ as supremum). $\;$This suggest to take $\vartheta=\frac{2C\epsilon}{\lambda^{k_0}}<\epsilon$ .

Now, let us analyze the case when the components have measure greater than $\widetilde{\delta}$.$\;$In fact, the idea is the same but we have that if $\widetilde{V}_n$ is component with measure greater or equal than $\delta$, then $f^{k_0} (\widetilde{V}_n)$ will split in at most $2^{k_0}$ components inside $W_{n+k_0}.$
Call $a_1,a_2,\dots,a_{2^{k_0-1}}$ the points that visit $c.$
Arguing as in the first case
the first case we see that
the measure of $\{t\;:\; d(c_n(t),a_i)\leq \frac{\epsilon}{\lambda^{k_0}} \}$ is comparable to
$$\mbox{Leb}(\{t\;:\; \Gamma_n(t) \leq \frac{\epsilon}{\lambda^{k_0}} \}).$$
Therefore
$$
\mbox{Leb}(\{t\;:\;c_{n+k_0}(t)\in \widetilde{V}_{n+k_0,1}\cup \cdots \widetilde{V}_{n+k_0,2^{k_0}} \mbox{ and } \Gamma_{n+k_0}(t)<\epsilon \}) $$
$$\leq \frac{2^{k_0}C\epsilon}{\lambda^{k_0} \delta} \mbox{Leb}(\{t: c_n(t)\in \widetilde{V}_n\}. $$
Summing over components we get
$$
\mbox{Leb}(\{t\;:\;\Gamma_{n+k_0}(t)<\epsilon \text{ and } c_n(t) \text{ is in a long component}\}
\leq M\eps  $$
where $ M=\frac{2^{k_0}C\epsilon}{\lambda^{k_0} \delta}.$

Combining the two cases we get
$$Z_{n+k_0} \leq Z_n\vartheta+M $$
as claimed.

\end{proof}

 Recalling that $t \in [a,b]$ is fixed and $h>0$ is such that $t+h \in [a,b]$,  define $\bar{I}_h=[c-hJ_t(c),c+hJ_t(c)]$.  Also, define $n_1=n_1(h)$ such that there exists $s_1 \in [t,t+h]$ so that

    $$f_{s_1} ^{-n_1}\bar{I}_h \cap \bar{I}_h \neq \emptyset,$$ and

    $$f_{s}^{-n}\bar{I}_h \cap \bar{I}_h = \emptyset,$$ for all $n< n_1$ and for all $s \in [t,t+h]$.

\begin{lemma}\label{lemmaboundeddistortion}
    For all $s_1,s_2 \in [t,t+h]$ and for all $n \leq n_1$

    \begin{equation} \label{boundeddistortion}
    \frac{1}{C} \leq  \frac{|f_{s_1}^n(\bar{I}_h) |}{|f_{s_2}^n(\bar{I}_h) |}  \leq C.
    \end{equation} and

    \begin{equation}\label{measurefs1}
    \frac{1}{\widetilde{C}} \leq \frac{|f_{s_1}^n(\bar{I}_h) |}{|c_n(t+h)-c_n(s_1) | }  \leq \widetilde{C}.
     \end{equation}

\end{lemma}


\begin{lemma}\label{lemmaunionsum}
$$\bigg|\int_{\bigcup_{k=1}^n f_{t+h}^{-k}(\bar{I}_h)} \psi(x)\;dx - \sum_{k=1}^n \int_{f_{t+h}^{-k}(\bar{I}_h)} \psi(x)\;dx\bigg| \leq \|\psi\|_{\infty} \sum_{k_1,k_2}^n |L_{k_1} \cap L_{k_2}|.$$
\end{lemma}


Note that

$$n_1^{-m} \leq |c_{n_1}(s)-c|\leq |c_{n_1}(s)-c_{n_1}(s_1) | + |c_{n_1}(s_1)-c|.$$ Up to a constant, the first term in the right side is bounded by $|f^{n_1}_{s_1}(\bar{I}_h) |$  by Lemma $\eqref{lemmaboundeddistortion}$, and the second term is also bounded by $|f^{n_1}_{s_1}(\bar{I}_h) |$ by definition of $s_1$. Then

\begin{equation}\label{n1alam}
n_1^{-m} =O(|f^{n_1}_{s_1}(\bar{I}_h)|)
\end{equation}

\vspace{.3cm}
Since

$$ \lambda^{n_1} h\leq |f^{n_1}_{s_1}(\bar{I}_h)| \leq 1, $$ we have that $n_1\leq \tilde{C}_1 |\log h|.\;$ Hence, the latter along with (\ref{n1alam}) gives

\begin{equation}
\label{HILarge}
|f^{n_1}_{s_1}(\bar{I}_h) |\geq (\tilde{C}_1|\log h|)^{-m}
\end{equation}

Hence, if $\Lambda=\sup_{x, t} Df(x,t)$ then

$$(\tilde{C}_1 |\log h|)^{-m}=O( h\Lambda^{n_1})$$ and so

\begin{eqnarray*}
 n_1 &\geq& \frac{|\log (h |\tilde{C}_1 \log h|^m)|}{\log\Lambda}
\end{eqnarray*}

Now, since $h$ is assume to be sufficiently small

\begin{equation}
\frac{|\log{h|\log{h}|^m}|}{\log\Lambda} \geq R|\log{h}|,
\end{equation} where $0<R\ll 1$.$\;$Then,

$$  n_1 \geq \frac{|\log (h |\tilde{C}_1 \log h|^m)|}{\log\Lambda}\geq R |\log{h}|,$$ where $\tilde{C}_2=\frac{\tilde{C}_1}{\Lambda}$.$\;$Therefore,

\begin{equation}\label{n1geqlog}
n_1 \geq R |\log{h}|.
\end{equation}

Thus, as $h$ decreases, $n_1(h)$ grows up tending to infinity.   Moreover, \eqref{n1geqlog} allows us to obtain the following estimate.

\begin{lemma}\label{lemmaeta} There exists $0<\eta <1$ such that
$$ \sum_{k_1,k_2=1}^n |L_{k_1} \cap L_{k_2}| \leq Cn^2h^{1+\eta}.$$
\end{lemma}

\begin{lemma}\label{IntegralOverL}
    Let $L$ be an interval such that $|L|=O(h)$ and let $m \in \mathbb{N}$.  If $\phi,\psi \in BV[0,1]$, then

    \begin{equation}
    \sum_{k=0}^{m} \int_L \phi(f^k_{t+h})\psi(x)dx=O(|L||\log|L||).
    \end{equation}

\end{lemma}

The details of the proofs of the lemmas above are shown in the Appendix.\\


We are interested in the limit

$$\limsup_{h \downarrow 0} \frac{\Gamma(t+h)-\Gamma(t)}{h\sqrt{|\log(h)\log\log|\log(h)||}}.$$

As mentioned before, we have that

$$\Gamma(t+h)-\Gamma(t) = -h\int_0^1\phi(f_t^ny)R_{t,n}(y)dy+\int_{[0,1]\backslash A_{h,n}} \phi(f_{t+h}^nx)dx - \int_{[0,1]\backslash B_{h,n}} \phi(f_t^ny)dy+O(h^2n^2). $$

By Lemma \ref{IntegralRn}, $-h\int_0^1\phi(f_t^ny)R_{t,n}(y)dy+O(h^2n^2)$ is negligible when dividing by \\
$h\sqrt{|\log(h)\log\log|\log(h)||}$, then we have to focus on the limit

\begin{equation}
\limsup_{h \downarrow 0} \frac{\int_{[0,1]\backslash A_{h,n}} \phi(f_{t+h}^nx)dx - \int_{[0,1]\backslash B_{h,n}} \phi(f_t^ny)dy}{h\sqrt{|\log(h)\log\log|\log(h)||}}.
\end{equation}

We already know that

\begin{equation}
\int_{[0,1]\backslash A_{h,n}} \phi \circ f_{t+h}^n(x)dx = \int_{\bigcup_{k=1}^nf_{t+h}^{-k}I_{n-k}} \phi \circ f_{t+h}^n(x)dx.
\end{equation}

Let $U_{n,k}=\bigcup_{k=1}^nf_{t+h}^{-k}I_{n-k}, V_{n,k}=\bigcup_{k=1}^nf_{t+h}^{-k}\bar{I}, P_{n,k}=\bigcup_{k=1}^n\bigg( f_{t+h}^{-k}I_{n-k} \backslash f_{t+h}^{-k}\bar{I}\bigg)$ and \\
$Q_{n,k}=\bigcup_{k=1}^n\bigg(f_{t+h}^{-k}\bar{I}\backslash  f_{t+h}^{-k}I_{n-k} \bigg).\;\;\;$Then,

$$\bigg| \int_{U_{n,k}}\phi \circ f_{t+h}^n(x)dx-\int_{V_{n,k}}\phi \circ f_{t+h}^n(x)dx  \bigg| =
 \bigg|\int_{P_{n,k}}\phi \circ f_{t+h}^n(x)dx-\int_{Q_{n,k} }\phi \circ f_{t+h}^n(x)dx \bigg| \leq$$
$$\leq \int_{P_{n,k}}|\phi \circ f_{t+h}^n(x)|dx + \int_{Q_{n,k} }|\phi \circ f_{t+h}^n(x)dx | \leq$$
$$\leq \sum_{k=1}^n \int_{f_{t+h}^{-k}I_{n,k} \backslash f_{t+h}^{-k}\bar{I} } |\phi \circ f_{t+h}^n(x)|dx + \int_{f_{t+h}^{-k}\bar{I} \backslash f_{t+h}^{-k}I_{n,k} } |\phi \circ f_{t+h}^n(x)dx | =$$
$$= \sum_{k=1}^n \int_{I_{n,k} \backslash \bar{I} } |\phi \circ f_{t+h}^n(x)|\mathcal{L}_{t+h}^kdx + \int_{\bar{I} \backslash I_{n,k} } |\phi \circ f_{t+h}^n(x)\mathcal{L}_{t+h}^kdx | \leq $$
$$\leq \sum_{k=1}^n \breve{C}\bigg(|I_{n-k}\backslash \bar{I}|+ |\bar{I} \backslash I_{n-k} |\bigg),$$ where $C$ comes from bounding $\phi$ and $\mathcal{L}_{t+h}^k$ (the latter is bounded by \eqref{StableSpectralGap}).

Now, note that $||I_{n-k}\backslash \bar{I} |=|\bar{I} \backslash I_{n-k} | = O(h).\;$Therefore we have

\begin{equation}
  \int_{\bigcup_{k=1}^n f_{t+h}^{-k}I_{n-k}} \phi\circ f_{t+h}^n(x)dx=\int_{\bigcup_{k=1}^nf_{t+h}^{-k}\bar{I} }\phi\circ f_{t+h}^n(x)dx + O(h).
\end{equation}

Since $\frac{h}{h\sqrt{|\log(h)\log\log|\log(h)||}} \to 0$ as $h\downarrow 0$, we can work with

$$\int_{\bigcup_{k=1}^nf_{t+h}^{-k}\bar{I} } \phi\circ f_{t+h}^n(x)dx$$ instead of $  \int_{\bigcup_{k=1}^n f_{t+h}^{-k}I_{n-k}} \phi\circ f_{t+h}^n(x)dx$.

\vspace{.3cm}

By Lemmas \eqref{lemmaunionsum} and \eqref{lemmaeta}, since $n^2h^{1+\eta}$ is negligible when dividing it by $h\sqrt{|\log(h)\log\log|\log(h)||}$, instead of working with $\int_{\bigcup_{k=1}^nf_{t+h}^{-k}\bar{I} }$, we can just focus on studying

\begin{equation}\label{focusedonft}
\sum_{k=1}^n  \frac{\int_{f_{t+h}^{-k}\bar{I}}\phi\circ f_{t+h}^n(x)dx}{h\sqrt{|\log(h)\log\log|\log(h)||}}.
\end{equation}

\vspace{.3cm}

Similarly, instead of working with $\int_{\bigcup_{k=1}^nf_t^{-k}\bar{I} }$, we can just focus on studying

\begin{equation}\label{focusedonf}
\sum_{k=1}^n  \frac{\int_{f_t^{-k}\bar{I}}\phi\circ f_t^n(x)dx}{h\sqrt{|\log(h)\log\log|\log(h)||}}.
\end{equation}

\subsection{Proof of Main Result}

    \begin{proof}[Proof of Theorem \ref{MainResult}]

    Let us recall that we shall assume that $\phi$ is of zero mean with respect to $\mu_t$.

    With the above in mind, we have to study the limit

    \begin{equation}\label{limitft}
    \Phi_1(t)=\limsup_{h \downarrow 0} \sum_{k=0}^n \frac{\int_{f_{t+h}^{-k}\bar{I}}\phi\circ f_{t+h}^n(x)dx}{h\sqrt{|\log(h)\log\log|\log(h)||}}
    \end{equation} and

    \begin{equation}\label{limitf}
    \Phi_2(t)= \limsup_{h \downarrow 0} \sum_{k=0}^n \frac{\int_{f_{t}^{-k}\bar{I}}\phi\circ f_t^n(x)dx}{h\sqrt{|\log(h)\log\log|\log(h)||}}
    \end{equation}

    Let us start with $(\ref{limitft})$.

    \vspace{.3cm}

    The derivative of $\rho_s$ equals a function $\rho_{1,s} \in BV[0,1]$ almost everywhere (\cite{Bal1}, \cite{CoDo}).  In particular, $\rho_{1,s}$ is bounded, hence, using  that $\rho_s$ is continuous at $c$ (since $c$ is not periodic) and its regular part \footnote{Recall that since $\rho_s \in BV$, it can be written as the sum of two functions, namely, the saltus part which is a sum of pure jumps
   and the regular part which is absolutely continuous.} is absolutely continuous (see \cite{CoDo}), we have that on $\bar{I}_h$ $$\rho_{t+h}(x)=\rho_{t+h}(c)+O(h).$$  Then,

$$    \sum_{k=0}^n \int_{f_{t+h}^{-k}\bar{I}_h}\phi\circ f_{t+h}^n(x)dx = \sum_{k=0}^n \int_{\bar{I}_h}\phi\circ f_{t+h}^{n-k}(x)\mathcal{L}_{t+h}^kdx $$
    $$= \sum_{k=0}^n \int_{\bar{I}_h}\phi\circ f_{t+h}^k(x)\mathcal{L}_{t+h}^{n-k}dx
    = \sum_{k=0}^n \int_{\bar{I}_h}\phi\circ f_{t+h}^k(x)\mathcal{L}_{t+h}^{n-k}dx $$
    $$= \sum_{k=0}^n \int_{\bar{I}_h}\phi\circ f_{t+h}^k(x)\rho_{t+h}dx + \sum_{k=0}^n \int_{\bar{I}_h}\phi\circ f_{t+h}^k(x)O(\theta^{n-k})dx$$
    $$\leq \sum_{k=0}^n \int_{\bar{I}_h}\phi\circ f_{t+h}^k(x)\rho_{t+h}dx +  |\bar{I}_h| \|\phi\| \sum_{k=0}^{\infty} O(\theta^k)\hspace{1.5cm} $$
    $$\leq \sum_{k=0}^n \int_{\bar{I}_h}\phi\circ f_{t+h}^k(x)\rho_{t+h}dx + O(h)
    = \sum_{k=0}^n \bigg[\int_{\bar{I}_h}\phi\circ f_{t+h}^k(x)dx\bigg]\rho_{t+h}(c) + O(h) .$$

    Thus,

    \begin{equation}\label{Phi1Reduced}
    \sum_{k=0}^n \int_{f_{t+h}^{-k}\bar{I}_h}\phi\circ f_{t+h}^n(x)dx = \sum_{k=0}^n \bigg[\int_{\bar{I}_h}\phi\circ f_{t+h}^k(x)dx\bigg]\rho_{t+h}(c) + O(h).
    \end{equation}

    Define $\widehat{I}_h=f^{n_1-1}\bar{I}_h$.$\;$For $n\geq n_1$.$\;$Since $f_{t+h}^{n_1-1}$ is $1-1$ (by definition of $n_1$), we have

    \begin{eqnarray*}
    \sum_{k=n_1+1}^{\lfloor |\log(h)| \rfloor}\int_{\bar{I}_h}\phi(f_{t+h}^k)\;dx &\leq& \sum_{n=1}^{\lfloor |\log(h)| \rfloor}\int_{\bar{I}_h}\phi(f_{t+h}^k)\;dx \\
    &=& \sum_{k=1}^{\lfloor |\log(h)| \rfloor}\bigg[\int_{\widehat{I}_h}\phi(f_{t+h}^{k-n_1+1}y)\frac{Df_{t+h}^{n_1-1}(c)}{Df_{t+h}^{n_1-1}(f_{t+h}^{-(n_1-1)}(y))}\;dy \bigg]\frac{1}{Df_{t+h}^{n_1-1}(c)}\\
    &=&O(\frac{|\widehat{I}_h|}{Df_{t+h}^{n_1-1}(c)} |\log|\widehat{I}_h| |),\\
    \end{eqnarray*} where the estimate above comes from Lemma $\ref{IntegralOverL}$.

    \vspace{.3cm}

    Using bounded distortion, we have

    \begin{eqnarray*}
    \frac{|\widehat{I}_h|}{Df_{t+h}^{n_1-1}(c)} &=& \frac{1}{Df_{t+h}^{n_1-1}(c)}\int_{\bar{I}_h}|Df_{t+h}^{n_1-1}(x)|\;dx\\
    &=& O(|\bar{I}_h|).\\
    \end{eqnarray*}

    By $\eqref{HILarge}$, $ |\widehat{I}_h| \geq (\tilde{C}_1|\log(h)|)^{-m}$.$\;$Thus

    \begin{eqnarray*}
    \log|\widehat{I}_h| &\geq& \log((\tilde{C}_1|\log(h)|)^{-m})\\
    &=& -m\log(\tilde{C}_1|\log(h)|).\\
    \end{eqnarray*}

    This implies $|\log|\widehat{I}_h|| \leq m\log(\tilde{C}_1|\log(h)|)$

Thus, we finally obtain

$$ \sum_{k=1}^{\lfloor |\log(h)| \rfloor} \int_{f_{t+h}^k(\bar{I})}\phi\circ f_{t+h}^n(x)\;dx =\bigg[\sum_{k=1}^{\lfloor |\log(h)| \rfloor} \int_{\bar{I}_h}\phi(f_{t+h}^kx)\;dx\bigg]\rho_{t+h}(c)+O(h)$$
$$= \bigg[\sum_{k=1}^{n_1} \int_{\bar{I}_h}\phi(f_{t+h}^kx)\;dx\bigg]\rho_{t+h}(c)+\bigg[\sum_{k=n_1}^{\lfloor |\log(h)| \rfloor} \int_{\bar{I}_h}\phi(f_{t+h}^kx)\;dx\bigg]\rho_{t+h}(c) +O(h)$$
$$= \bigg[\sum_{k=1}^{n_1} \int_{\bar{I}_h}\phi(f_{t+h}^kx)\;dx\bigg]\rho_{t+h}(c)+O(h\log|\log(h)|) )$$

    Since $\phi \in$ Lip$[0,1]$, we have that for $1 \leq k \leq n_1$

    \begin{eqnarray*}
    \phi(f_{t+h}^k(x))&=&\phi(f_{t+h}^k(c))+O(|f_{t+h}^k\bar{I}_h |) \\
    &=& \phi(f_{t+h}^k(c))+O(|\widehat{I}_h |).
    \end{eqnarray*} Then,

    \begin{equation}\label{EstimatePhi}
        \phi(f_{t+h}^k(x))=\phi(f_t^k(c))+O(h).
    \end{equation}

    Hence

    \begin{eqnarray*}
    \sum_{k=1}^{\lfloor |\log(h)| \rfloor} \int_{f_{t+h}^k(\bar{I}_h)}\phi\circ f_{t+h}^n(x)\;dx&=&\bigg[\sum_{k=1}^{n_1} \int_{\bar{I}_h}\phi(f_{t+h}^kx)\;dx\bigg]\rho_{t+h}(c)+O(h\log|\log(h)|))\\
    &=& \bigg[\rho_{t+h}(c)|\bar{I}_h| \sum_{k=1}^{n_1} \phi(f_{t+h}^k(c)) \bigg]+O(h\log|\log(h)|))\\
    &=& \bigg[2hJ_t(c)\rho_{t+h}(c) \sum_{k=1}^{n_1} \phi(f_{t+h}^k(c)) \bigg]+O(h\log|\log(h)|)).
    \end{eqnarray*}

    Therefore, we have that

    \begin{equation}\label{hloglogh}
     \sum_{k=1}^{\lfloor |\log(h)| \rfloor} \int_{f_{t+h}^k(\bar{I})}\phi\circ f_{t+h}^n(x)\;dx
         \end{equation}
         $$= \bigg[2hJ_t(c)\rho_{t+h}(c) \sum_{k=1}^{n_1} \phi(f_{t+h}^k(c)) \bigg]+O(h\log|\log(h)|)) . $$

    Now, define $n_2=n_2(h)$ as the smallest number such that

    \begin{equation}\label{n2}
    |Df^{n_2}_{t+h}(c)||\bar{I}_h| \geq 1,
    \end{equation} where $Df^{n_2}_{t+h}(c)$ must be understood as $\min \{Df^{n_2}_{t+h}(c+),Df^{n_2}_{t+h}(c-) \}$, with $Df^{n_2}_{t+h}(c\pm )$ the side derivatives of $c$.

    We claim that $n_2-n_1 \leq C \log|\log(h)|$. Indeed write
    $$|Df^{n_2}_{t+h}(c)|= |Df_{t+h}^{n_2-1}(c_1)Df_{t+h}(c)|.$$
    Using the definition of    $n_2$ we have that $|Df_{t+h}^{n_2-1}(c_1)||\bar{I}_h|\leq 1$ so

    \begin{equation}\label{boundedaboveDfn23}
    |Df_{t+h}^{n_2}(c)||\bar{I}_h| \leq C_2,
    \end{equation} where $C_2=\displaystyle \max_{x} |Df(x)|$.

    \vspace{.3cm}

    Note that

    \begin{eqnarray*}
    |Df_{t+h}^{n_2}(c)||\bar{I}_h|&=&|Df_{t+h}^{n_2-n_1+1}(f_{t+h}^{-n_1+1}(c))||Df_{t+h}^{n_1-1}(c)||\bar{I}_h|\\
    &\leq& \Lambda^{n_2-n_1+1}|Df_{t+h}^{n_1-1}(c)||\bar{I}_h|. \\
    \end{eqnarray*}

    As before, using bounded distortion, we have that

    $$|Df_{t+h}^{n_1-1}(c)||\bar{I}_h| \geq C_1 |\widehat{I}_h|.$$

    Hence,

    \begin{equation}\label{boundedbelowDfn2}
    |Df_{t+h}^{n_2}(c)||\bar{I}_h| \geq C_1\Lambda^{n_2-n_1+1}|\widehat{I}_h|
    \end{equation}

    Using \eqref{boundedaboveDfn23} and \eqref{boundedbelowDfn2},  we finally obtain

    \begin{equation}
    C_1\Lambda^{n_2-n_1+1}|\widehat{I}_h| \leq C_2,
    \end{equation}which implies

    \begin{equation}\label{n2n1}
    n_2-n_1+1 = O(|\log|\widehat{I}_h| |).
    \end{equation}

    Back to $(\ref{hloglogh})$, we can decompose it as

    \begin{eqnarray*}
     \sum_{k=1}^{\lfloor |\log(h)| \rfloor} \int_{f_{t+h}^k(\bar{I})}\phi\circ f_{t+h}^n(x)\;dx&=& \bigg[2hJ_{t+h}(c)\rho_{t+h}(c) \sum_{n=1}^{n_1} \phi(f_{t+h}^n(c)) \bigg]+O(h\log|\log(h)|)) \\
    \end{eqnarray*}
    $$= \bigg[2hJ_{t+h}(c)\rho_{t+h}(c) \sum_{n=1}^{n_2} \phi(f_{t+h}^n(c)) + \sum_{n=n_2+1}^{n_1} \phi(f_{t+h}^n
    (c)) \bigg]+ O(h\log|\log(h)|)).$$

    Using $(\ref{n2n1})$, we have that

    $$     \sum_{k=1}^{\lfloor |\log(h)| \rfloor} \int_{f_{t+h}^k(\bar{I})}\phi\circ f_{t+h}^n(x)\;dx= \bigg[2hJ_{t+h}(c)\rho_{t+h}(c) \sum_{n=1}^{n_2} \phi(f_{t+h}^n(c)) \bigg]+O(h\log|\log(h)|))$$

    \vspace{.3cm}
    Thus,

    \begin{eqnarray*}
   \Phi_1(t) &=&   \limsup_{h\downarrow0} \frac{   \sum_{k=1}^{\lfloor |\log(h)| \rfloor} \int_{f_{t+h}^k(\bar{I})}\phi\circ f_{t+h}^n(x)\;dx}{h\sqrt{|\log(h)\log\log|\log(h)||}} \\
   &=& 2\rho_t(c)J_{t}(c) \limsup_{h \downarrow0} \frac{\sum_{n=1}^{n_2}\phi(f_{t+h}^n(c))}{h\sqrt{|\log(h)\log\log|\log(h) ||}}\\
    &=&2\rho_t(c)J_{t}(c) \limsup_{h \downarrow0} \frac{\sum_{n=1}^{n_2}\phi(f_{t+h}^n(c))\sqrt{n_2\log\log n_2}}  {h\sqrt{|\log(h)\log\log|\log(h)||}\sqrt{n_2\log\log n_2}} \\
    &=& 2\rho_t(c)J_{t}(c) \limsup_{h\downarrow0}\frac{\displaystyle\sum_{n=1}^{n_2}\phi(f_{t+h}^n(c))  }{\sqrt{n_2\log\log n_2}} \limsup_{h \downarrow0}\frac{\sqrt{n_2\log\log n_2}}{\sqrt{|\log(h)|\log\log|\log(h)|}} .
    \end{eqnarray*}

    Note that, as $h \downarrow 0$, $n_2$  goes to $\infty$. Then, by \cite{Sch2}, for almost all $t$,

    $$\limsup_{h \downarrow 0}\frac{\displaystyle\sum_{n=1}^{n_2}\phi(f_{t+h}^n(c))  }{\sqrt{n_2\log\log n_2}} = \limsup_{n_2 \to \infty} \frac{\displaystyle\sum_{n=1}^{n_2}\phi(f_{t+h}^n(c))  }{\sqrt{n_2\log\log n_2}}=\sqrt{2}\sigma_t(\phi).$$


    For the limit $\limsup_{h \downarrow0}\frac{\sqrt{n_2\log\log n_2}}{\sqrt{|\log(h)|\log\log|\log(h)|}}$, we will prove that
    $\limsup_{h \downarrow0} \frac{n_2}{|\log(h)|}$ converges which implies the convergence of the limit we want.$\;$For this, note that

    $$\log |Df_{t+h}^{n_2}(c)|=\sum_{j=0}^{n_2-1} \log |Df_{t+h}(f_{t+h}^j(c))|$$

    Then,

    $$\frac{\log |Df_{t+h}^{n_2}(c)|}{n_2} = \frac{\sum_{j=0}^{n_2-1}\log |Df_{t+h}(f_{t+h}^j(c))|}{n_2}.$$

    By Theorem $1.2$ in \cite{Sch1}, for almost all $t$, the sequence $\frac{\sum_{j=0}^{n-1}\log |Df_{t+h}(f_{t+h}^j(c))|}{n}$ converges to $\int \log |Df_{t+h}(x)|\;d\mu_{t+h}$ and so does its     subsequence $\frac{\sum_{j=0}^{n_2-1}|Df_{t+h}(f_{t+h}^j(c))|}{n_2}$ so $\frac{\log |Df_{t+h}^{n_2}(c)|}{n_2}$ converges as $n_2 \to \infty$.

    \vspace{.3cm}

    Also, as we already saw, $|Df^{n_2}_{t+h}(c)||\bar{I}_h|$ is bounded by below (by $1$ by definition) and by above for some constant $C$.$\;$Then

    $$ 1 \leq |Df^{n_2}_{t+h}(c)||\bar{I}_h| \leq C, $$ which implies

    $$ \frac{ |\log(h)|}{n_2} \leq \frac{\log |Df^{n_2}_{t+h}(c)|}{n_2} \leq \frac{|\log{\frac{h}{C}}|}{n_2},$$ and because $\frac{\log |Df_{t+h}^{n_2}(c)|}{n_2}$ converges so does $\frac{|\log(h)|}{n_2}$ as $h \to 0$ to the same limit.  Then

    $$\limsup_{h \downarrow 0 } \frac{n_2(h)}{|\log(h)|}=\bigg( \int \log |Df_t(x)|d\mu_t\bigg)^{-1}.$$In particular, this implies that $\limsup_{h \downarrow 0} \frac{\log\log n_2}{\log\log|\log(h)|}=1$.  Hence, for almost all $t$,

    \begin{equation}\label{SecondLimitN2}
      \limsup_{h \downarrow 0} \frac{\sqrt{n_2\log\log n_2}}{\sqrt{|\log(h)||\log\log|\log(h)|}}=\bigg( \int \log |Df_t(x)|d\mu_t\bigg)^{-1/2}.
    \end{equation}

    \vspace{.3cm}

    Therefore, we finally conclude that $\Phi_1(t)$ exists for almost all $t$ and equals

    $$\Phi_1(t)=2\sqrt{2}\rho_t(c)J_{t}(c)\sigma_t(\phi)\bigg(\int \log |Df_t(x)|\;d\mu_t\bigg)^{-1/2}.$$

In order to analyze $(\ref{focusedonf})$, we need to work on the limit $\Phi_2(t)$.$\;$For this, note that \eqref{Phi1Reduced} remains true if we  replace $f_t$ instead of $f_{t+h}$.  Since $c$ is periodic for the expanding map $f_t$, we can use \eqref{lemmaunionsum}, \eqref{lemmaeta}, and the assumption that $\phi$ is of zero mean with respect to $\mu_t$ to get that $\Phi_2(t)$ is zero.\\

Therefore, for almost all $t$,

 $$\limsup_{h\downarrow0} \frac{\Gamma(t+h)-\Gamma(h)}{h\sqrt{|\log(h)|\log\log|\log(h)| |}} = 2\sqrt{2}\rho_t(c)J_{t}(c)\sigma_t(\phi)\bigg(\int \log |Df_t(x)|\;d\mu_t\bigg)^{-1/2}$$ as claimed.
\end{proof}


\section{Appendix}

\begin{proof}[Proof of Lemma \ref{lemmaboundeddistortion}]
    Define $d_n=\sqrt{(s_1-s_2)^2+(c_n(s_1)-c_n(s_2))^2  }.\;$ By the Mean Value Theorem, there exists  $\widetilde{s}$ between $s_1$ and $s_2$ such that

    $$\frac{c_n(s_1)-c_n(s_2)}{s_1-s_2} = \frac{\partial c_n }{\partial s}(\widetilde{s}).$$

    By the Chain Rule,

    $$\frac{\partial c_n }{\partial s}(\widetilde{s}) = J_n(c(\widetilde{s}))Df^n_{\widetilde{s}}(c_1(\widetilde{s})) $$

    Since $Df^n_{\widetilde{s}}(c_1(\widetilde{s})) \geq \lambda^n$ and $J_n$ converges as $n$ goes to infinity
    $$\bigg|\frac{\partial c_n }{\partial s}(\widetilde{s})\bigg| \geq \lambda^n_{\widetilde{s}}C_{12}.$$
   Hence $\bigg(\frac{\partial c_n }{\partial s}(\widetilde{s}) \bigg)^{-1} = O(\lambda^{-n})$. So

    \begin{eqnarray*}
    {\sqrt{1+\left(\frac{\partial c_n}{\partial s}\right)^2(\widetilde{s})}}&=& \left|\frac{\partial c_n(\widetilde{s})}{\partial s}\right|\sqrt{1+\left(\frac{\partial c_n}{\partial s}\right)^{-2}(\widetilde{s})}\\
    &=&\left|\frac{\partial c_n}{\partial s}\right|(\widetilde{s})\left[1+O(\lambda^{-2n})\right].
    \end{eqnarray*}

    Then,

    \begin{eqnarray*}
    d_n&=&|s_1-s_2|\sqrt{1+\left(\frac{\partial c_n}{\partial s}\right)^2(\widetilde{s})}\\
    &=&|s_1-s_2| \left|\frac{\partial c_n}{\partial s}\right|(\widetilde{s})\left[1+O(\lambda^{-2n})\right]\\
    &=&|c_n(s_1)-c_n(s_2)|\left[1+O(\lambda^{-2n})\right]
    \end{eqnarray*}

    Thus,

    \begin{equation}\label{dnequality}
    d_n=|c_n(s_1)-c_n(s_2)|\left[1+O(\lambda^{-2n})\right]  \end{equation}

    We claim that

    \begin{equation}\label{inequalitydn}
    d_{n+1} \geq (\lambda-\delta)d_n
    \end{equation}

    In fact, by $(\ref{dnequality})$, this is the same as proving

    \begin{equation}
    |c_{n+1}(s_1)-c_{n+1}(s_2)| \geq (\lambda-\delta) |c_n(s_1)-c_n(s_2)|.
    \end{equation}
    Since
    $$\frac{ |c_{n+1}(s_1)-c_{n+1}(s_2)|}{|c_n(s_1)-c_n(s_2)|}= \left|\frac{\partial c_{n+1}}{\partial c_n} (\widetilde{\widetilde{s}}) \right|$$
   it suffices to show that
    \begin{equation}\label{partialcn+1overcn}
    \left|\frac{\partial c_{n+1}}{\partial c_n} (\widetilde{\widetilde{s}}) \right| \geq \lambda -\delta
    \end{equation} so let us prove this last inequality.

    \vspace{.3cm}

    Since $\left|\frac{\partial c_n}{\partial s}\right| \geq D\lambda^n $, in particular $\left|\frac{\partial c_n}{\partial s}\right| \neq 0$, so by the Implicit Function Theorem, $s=s(c_n)$ and

    $$\left| \frac{\partial s}{\partial c_n} \right| = \left| \frac{1}{\frac{\partial c_n}{\partial s} } \right| \leq \frac{D}{\lambda^n} $$

    Since $c_{n+1}=f_{s(c_n)}(c_n)$, by using Chain Rule,

    \begin{eqnarray*}
    |Df_s(c_n)| &\leq& \left| \frac{\partial c_{n+1}}{\partial c_n}  \right| + \left| \frac{\partial c_n}{\partial s}(s)\frac{\partial s}{\partial c_n} \right| \\
    \end{eqnarray*}

    This implies

    \begin{eqnarray*}
    \left| \frac{\partial c_{n+1}}{\partial c_n}  \right| &\geq&     |Df_s(c_n)| - \left| \frac{\partial c_n}{\partial s}(s)\frac{\partial s}{\partial c_n} \right| \\
    &\geq& \lambda_s - \delta,\\
    \end{eqnarray*} where $\lambda^{-n}D<\delta \ll 1 .$
    \vspace{.3cm}

    Therefore, $(\ref{partialcn+1overcn})$ holds, which, as discussed, implies

    $$d_{n+1} \geq (\lambda-\delta)d_n.$$

    \vspace{.3cm}

    With the above in mind, if $D$ is the function on $[0,t] \times [0,1]$ defined by $D(s,x)=Df_s(x)$, then $\log \circ |D|$ is $P-$Lipschitz, for some constant $P$,  hence

$$
\log \frac{|Df^n_{s_1}(c)|}{|Df^n_{s_2}(c)|} \leq \sum_{k=0}^{n-1} \log|Df_{s_1}(c_k(s_1))| - \log|Df_{s_2}(c_k(s_2))| $$
$$
      \leq \sum_{k=0}^{n-1} P d_k
      \leq P \sum_{k=0}^{n-1} \frac{d_n}{(\lambda-\epsilon)^{n-k}}
      \leq \widetilde{P} d_n
      \leq \hat{C},$$
    where $\widetilde{P} = \sum_{j=1}^{\infty} \frac{1}{(\lambda-\epsilon)^{j}}$ (note that $d_n$ is bounded by 2).$\;$ Therefore, $\frac{|Df^n_{s_1}(c)|}{|Df^n_{s_2}(c)|}$ is bounded above by some constant $C_1$ and since $s_1$ and $s_2$ are arbitrary then they are exchangeable so the expression $\frac{|Df^n_{s_1}(c)|}{|Df^n_{s_2}(c)|}$ is also bounded by below by the reciprocal of $C_1$.

    \vspace{.3cm}

    Since

    $$\frac{|Df^n_{s_1}(x)|}{|Df^n_{s_2}(y)|} = \frac{|Df^n_{s_1}(x)| }{|Df^n_{s_1}(c) |} \frac{|Df^n_{s_1}(c) |}{|Df^n_{s_2}(c) |} \frac{|Df^n_{s_2}(c) |}{|Df^n_{s_2}(y) |},$$ using that $f_{s_1}$ and $f_{s_2}$ are functions of bounded distortion, we have that

    \begin{equation}
    \frac{1}{C} \leq \frac{|Df^n_{s_1}(x)|}{|Df^n_{s_2}(y)|} \leq C,
    \end{equation} where $C=C_1C_2C_3$ and $C_2$ and $C_3$ are the bounds for the distortion of $f_{s_1}$ and $f_{s_2}$ respectively.

    Hence, if $x \in \bar{I}_h$, we have that

    $$\frac{|\bar{I}_h|}{C} \leq \int_{\bar{I}_h} \frac{|Df_{s_2}^n(y)|}{ |Df_{s_1}^n(x) | }   dy \leq C|\bar{I}_h | . $$
    Therefore
$$
    \frac{|f^n_{s_1}\bar{I}_h|}{|f^n_{s_2}\bar{I}_h|} = \frac{\int_{\bar{I}_h} |Df^n_{s_1}(x)|dx}{\int_{\bar{I}_h} |Df^n_{s_2}(y)|dy}
    = \int_{\bar{I}_h}  \frac{1}{\int_{\bar{I}_h} \frac{|Df_{s_2}^n(y)|}{ |Df_{s_1}^n(x) | }   dy  } dx
    \leq \int_{\bar{I}_h} \frac{1}{ \frac{|\bar{I}_h|}{C} }  dx
    = C. $$

    Thus
$\frac{|f^n_{s_1}\bar{I}_h|}{|f^n_{s_2}\bar{I}_h|} \leq C.$
Similarly
$\frac{1}{C} \leq \frac{|f^n_{s_1}\bar{I}_h|}{|f^n_{s_2}\bar{I}_h|},$ and so \eqref{boundeddistortion} holds.

    To prove \eqref{measurefs1}, note that $\frac{J(c(0))}{J_n(c(s))}$ is bounded above and below by some constant $C_6$ and $\frac{1}{C_6}$ respectively (because as $t$ decreases so does $s$ and $n$ increases as well, so $J_n(c(s))$ converges to $J(c(0))$).  Then, using that $\frac{\partial c_n }{\partial s } (\widetilde{s})= J_n(c(\widetilde{s}))Df_{\widetilde{s}}^n(c_1(\widetilde{s}) $ (by the Chain Rule) we have

    \begin{eqnarray*}
     \frac{|f_{s_1}(\bar{I}_h)|}{|c_n(s_2)-c_n(s_1)|} &=& \frac{1}{|h||\frac{\partial c_n }{\partial s }|} \int_{\bar{I}_h} |Df_{s_1}^n(x)|dx\\
     &=& \frac{1}{|h||J_n(c(\widetilde{s}))Df_{\widetilde{s}}^n(c_1(\widetilde{s}) |}\int_{\bar{I}_h} |Df_{s_1}^n(x)|dx\\
     &=&
     \frac{1}{|h||J_n(c(\widetilde{s}))|} \int_{\bar{I}_h} \frac{|Df_{s_1}^n(x)|}{|Df_{\widetilde{s}}^n(c_1(\widetilde{s})) |}dx\\
     &=& O\bigg(\frac{|\bar{I}_h|}{|h||J_n(c(\widetilde{s}))|}\bigg) \\
     &=&O\bigg(\frac{|J_n(c)|}{|J_n(c(\widetilde{s}))|}\bigg)\\
     &=& O(1), \\
    \end{eqnarray*}where we use \eqref{boundeddistortion} to bound $\frac{|Df_{s_1}^n(x)|}{|Df_{\widetilde{s}}^n(c_1(\widetilde{s}) )|}$.

    Then,

    $$\frac{|f_{s_1}(\bar{I}_h)|}{|c_n(t)-c_n(s_1)|} \leq  \widetilde{C}  $$ for some constant $\widetilde{C}$.\\

    Similarly, we can prove that

    $$\frac{1}{\widetilde{C}} \leq \frac{|f_{s_1}(\bar{I}_h)|}{|c_n(t)-c_n(s_1)|}  .$$
   Thus we obtain \eqref{measurefs1}.
\end{proof}

\begin{proof}[Proof of Lemma \ref{lemmaunionsum}]
Set $L_k=f_{t+h}^{-k}(\bar{I}_h)$ and define $\widetilde{L}_k=L_k-\displaystyle \bigcup_{j<k}L_j\cap L_k$.$\;$Note that $\displaystyle \bigcup_{k=1}^n L_k = \bigcup_{k=1}^n \widetilde{L}_k$ and that $\widetilde{L}_{k_1} \cap \widetilde{L}_{k_2} = \emptyset $, for all $k_1,k_2$, in particular,

\begin{eqnarray*}
\int_{\bigcup_{k=1}^n L_k}\psi(x)\;dx&=& \int_{\bigcup_{k=1}^n \widetilde{L}_k} \psi(x)\;dx\\
&=& \sum_{k=1}^n \int_{\widetilde{L}_k}\psi(x)\;dx.\\
\end{eqnarray*}

Now, we can work with $\int_{\widetilde{L}_k}\psi(x)\;dx$ and bound $\sum_{k=1}^n \int_{\widetilde{L}_k}\psi(x)- \sum_{k=1}^n \int_{L_k}\psi(x)\;dx$. For this, since

\begin{eqnarray*}
 \int_{\widetilde{L}_k}\psi(x)\;dx &=& \int_{L_k}\psi(x)\;dx - \int_{L_k \backslash \widetilde{L}_k} \psi(x)\;dx \\
&=& \int_{L_k}\psi(x)\;dx - \int_{\bigcup_{j<k}L_j\cap L_k} \psi(x)\;dx \\
\end{eqnarray*} we have that

\begin{eqnarray*}
\bigg| \sum_{k=1}^n\int_{L_k}\psi(x)\;dx - \int_{\widetilde{L}_k}\psi(x)\;dx| &=& \sum_{k=1}^n \bigg| \int_{\bigcup_{j<k}L_j\cap L_k} \psi(x)\;dx \bigg| \\
&\leq& \sum_{k=1}^n \bigg|\bigcup_{j<k}L_j\cap L_k\bigg| \|\psi \|_{\infty} \\
&\leq& \sum_{k=1}^n \sum_{j<k}\bigg| L_j\cap L_k\bigg| \|\psi \|_{\infty} \\
&\leq& \sum_{k_1,k_2=1}^n \bigg| L_j\cap L_k\bigg| \|\psi \|_{\infty} \\
\end{eqnarray*}

Therefore,

\begin{eqnarray*}
\bigg| \sum_{k=1}^n\int_{L_k}\psi(x)\;dx - \int_{\widetilde{L}_k}\psi(x)\;dx|
&\leq& \sum_{k_1,k_2=1}^n \bigg| L_j\cap L_k\bigg| \|\psi \|_{\infty} \\
\end{eqnarray*}

\end{proof}

\begin{proof}[Proof of Lemma \ref{lemmaeta}]
  Again, set $L_k=f_{t+h}^{-k}(\bar{I}_h)$.  Let $1 \leq k_1,k_2 \leq n$.$\;$Without loss of generality, assume $k_1 \leq k_2$.$\;$Then, we can write $k_2=k_1+j$, for some $0\leq j \leq n-k_1$.$\;$Then, using the fact that $\rho_t$ is bounded below, we have

  \begin{eqnarray*}
  |L_{k_1} \cap L_{k_2}| &=& |L_{k_1} \cap L_{k_1+j} | \\
  &=& \int \chi_{\bar{I}_h}(f_{t+h}^{k_1}(x))\chi_{\bar{I}_h}(f_{t+h}^{k_1+j}(x))dx\\
  &\leq& C_1\int \rho_{t+h}(x)\chi_{\bar{I}_h}(f_{t+h}^k(x))\chi_{\bar{I}_h}(f_{t+h}^{k+j}(x))dx\\
  \end{eqnarray*}

  Since $\rho_{t+h}$ is invariant, $\int \rho_{t+h}(x)\chi_{\bar{I}_h}(f_{t+h}^k(x))\chi_{\bar{I}_h}(f_{t+h}^{k+j}(x))dx = \int \rho_{t+h}(x)\chi_{\bar{I}_h}(x)\chi_{\bar{I}_h}(f_{t+h}^{j}(x))dx$. Now using that $\rho_{t+h}$ is bounded from above, we have

  \begin{eqnarray*}
    |L_{k_1} \cap L_{k_2}| &=& |L_{k_1} \cap L_{k_1+j} | \\
  &\leq& C_1\int \rho_{t+h}(x)\chi_{\bar{I}_h}(f_{t+h}(x))\chi_{\bar{I}_h}(f_{t+h}^{j}(x))dx\\
  &=& C_2 \int \chi_{\bar{I}_h}(f_{t+h}(x))\chi_{\bar{I}_h}(f_{t+h}^{j}(x))dx \\
  \end{eqnarray*}

  If $j < n_1$ then $f_{t+h}^{-j}\bar{I}_h \cap \bar{I}_h = \emptyset,$ and consequently

  $$\int \chi_{\bar{I}_h}(f_{t+h}(x))\chi_{\bar{I}_h}(f_{t+h}^{j}(x))dx =0.$$

  If $j >n_1$ then, by \eqref{n1geqlog}, $\theta^{j} < \theta^{R|\log{t}|}$.$\;$Hence

  \begin{eqnarray*}
  \int \chi_{\bar{I}_h}(x)\chi_{\bar{I}_h}(f_{t+h}^j(x))dx &=& \int \chi_{\bar{I}_h}(y)\mathcal{L}_{t+h}^j(\chi_{\bar{I}_h})(y)dy \\
  &=& \int_{\bar{I}_h} \mathcal{L}_{t+h}^j(\chi_{\bar{I}_h})(y)dy \\
  &=& \int_{\bar{I}_h} |\bar{I}_h| \rho_{t+h}(y) + O(\theta^{j}) dy\\
  &\leq& \int_{\bar{I}_h} |\bar{I}_h| \rho_{t+h}(y)dy + \int_{\bar{I}_h}O(\theta^{j}) dy\\
  \end{eqnarray*}

  By \eqref{StableSpectralGap} and $|\bar{I}_h|=O(h)$, the first integral is of order $O(t^2)$.$\;$For the second integral, since $\theta^{j} \leq \lambda^{R|\log(h)|}$, we have that $\theta^{j} \leq h^{\eta}$, where $\eta=R\log(\theta^{-1})<1$ since $R\ll 1$.$\;$Then, the second integral is of order $O(h^{1+\eta})$.

  Therefore, $|L_{k_1} \cap L_{k_2}| \leq Ch^{1+\eta}$, for any $1\leq k_1 < k_2 \leq n$ and then

  $$\sum_{k_1 < k_2} |L_{k_1} \cap L_{k_2}| \leq Ch^{1+\eta}n^2.$$

\end{proof}

\begin{proof}[Proof of Lemma \ref{IntegralOverL}]

Suppose $m < l$, where $l=\lfloor |\log|L|\| \rfloor$.  It is not hard to see that

    $$\bigg|  \sum_{k=0}^m \int_L \phi(f^k_{t+h})\psi(x)\;dx\bigg| = O(|L||\log|L||).$$

Therefore, the statement holds if $m\leq l$.

Suppose now $m \geq l$ and write $m=l+r$, with $0\leq r \leq l-1$. Let$\Lambda=\max_{x,t} Df_t(x)$.

Note that if $f(t)=t\log(1/t)$ and $g(s)=s^{\log\Lambda } \log(1/s)$, then $g>f>0$ and $g'\gg f'$ near 0.  Hence, if $|L|=O(t)$, we have that $t\log(1/t) < |L|^{\log\Lambda}\log(1/|L|)$, or equivalently

\begin{equation}\label{BoundForllogt}
  t|\log(t)|<(\Lambda^r-1)l\Lambda^{-|\log|L||}.
\end{equation}

$$\bigg| \sum_{k=0}^m \int_L \phi(f^k_{t+h})\psi(x)\;dx\bigg| = \bigg| \sum_{k=0}^{l-1} \int_L \phi(f^k_{t+h})\psi(x)\;dx\bigg|+\bigg| \sum_{k=l}^m \int_L \phi(f^k_{t+h})\psi(x)\;dx\bigg|$$
 $$\sum_{k=l}^m \int_L \Lambda^k|L| +O(|L||\log|L||)= O((\Lambda^{m+1}-\Lambda^l)|L|).$$

 Since \eqref{BoundForllogt} implies that $(\Lambda^{m+1}-\Lambda^l)t|\log(t)|<l$, the statement also holds if $m\geq l$.

\end{proof}

\end{document}